\documentclass[a4paper,article]{amsart}
\usepackage{amsfonts}
\usepackage{amssymb}
\usepackage{amsthm}
\usepackage{enumerate}
\usepackage[english]{babel}
\usepackage{hyperref}
\usepackage[T1]{fontenc}
\usepackage[utf8]{inputenc}
\usepackage{caption}
\usepackage{mathtools}
\usepackage[usenames,dvipsnames]{pstricks}

\usepackage{color}

\theoremstyle{plain}
\begingroup
\theoremstyle{plain}
\newtheorem{theorem}{Theorem}[section]

\newtheorem{proposition}[theorem]{Proposition}

\theoremstyle{definition}
\newtheorem{definition}[theorem]{Definition}

\theoremstyle{remark}
\newtheorem{remark}[theorem]{Remark}
\theoremstyle{claim}
\newtheorem{claim}[theorem]{Claim}
\theoremstyle{definition}
\newtheorem{notation}[theorem]{Notation}
\endgroup

\theoremstyle{definition}
\theoremstyle{remark}

\mathchardef\emptyset="001F

\numberwithin{equation}{section}

\makeatletter
\def\Ddots{\mathinner{\mkern1mu\raise\p@
\vbox{\kern7\p@\hbox{.}}\mkern2mu
\raise4\p@\hbox{.}\mkern2mu\raise7\p@\hbox{.}\mkern1mu}}
\makeatother

\title[]
{Optimality of broken extremals.}
\author[A. A. Agrachev]{Andrei A. Agrachev}
\address[A. A. Agrachev]{SISSA, 34136 Trieste,  Italy; Steklov Mathematical Institute, 119991 Moscow, Russia}
\email[A. Agrachev]{agrachev@sissa.it}
\author[C. Biolo]{Carolina Biolo}
\address[Carolina Biolo]{SISSA, Via Bonomea 265, 34136 Trieste, Italy}
\email[Carolina Biolo]{cbiolo@sissa.it}
\date{}

\begin{document}

\begin{abstract}
In this paper we analyse the optimality of broken Pontryagin extremal for an $n$-dimensional affine control system with a control parameter, taking values in a $k$-dimensional closed ball. We prove the optimality of broken normal extremals when $n=3$ and the controllable vector fields form a contact distribution, and when the Lie algebra of the controllable fields is locally orthogonal to the singular locus and the drift does not belong to it. Moreover, if $k=2$, we show the optimality of any broken extremal even abnormal when the controllable fields do not form a contact distribution in the point of singularity.
\end{abstract}
\maketitle
\tableofcontents
\section{Introduction}
This paper is closely related to \cite{AB2} and \cite{AB}, where the authors study the local regularity of time-optimal controls and trajectories for the control system of the form:
\begin{equation}
\label{1,1}
\dot q=f_0(q)+\sum_{i=1}^ku_if_i(q),\quad q\in M,\ (u_1,\ldots,u_k)\in U,
\end{equation}
where $M$ is a smooth $n$-dimensional manifold, $U=\{u\in \mathbb{R}^k : ||u||\leq 1\}$ is the $k$-dimensional ball, and $f_0,\,f_1,\,\ldots,\,f_k$ are smooth\footnote{We work in $\mathcal{C}^\infty (M)$ category.} vector fields. We also assume that $f_1(q),\ldots,f_k(q)$ are linearly independent in the domain under consideration.

If $k=n$, then all extremals are smooth; otherwise they may be nonsmooth and there exists a vast literature dedicated to the case $k=1$. Some references can be found in paper \cite{AB}.

At \cite{AB} and \cite{AB2} the authors prove that with some generic conditions it is possible to avoid chattering phenomenon if $k<n$ and that the singularity must be isolated, moreover we denoted in which cases it is possible to find non smooth optimal trajectories. 

Actually, in that paper we did not claim that they exist. Indeed, via the Pontryagin maximum principle, we know that every time-optimal trajectory has a lift, called extremal, in $T^*M$. But, on the other hand it is not guaranteed that given any extremal its projection on $M$ is time-optimal: even though we have found extremals through the singular locus $\Lambda$ that projects in piece-wise smooth trajectories, non necessarily those trajectories are time-optimal.

The optimality of the projection of any extremal is guaranteed only if we consider a linear control system, satisfying Kalman's Criterion:
$$
\mathrm{rank}\{B,AB,\hdots,A^{n-1}B \}=n,
$$
and put the final point in a equilibrium. It is true due to the fact that the uniqueness of the time-optimal solution and the uniqueness of the extremal hold.\\

In this paper we are going to discuss the optimality of the projections of the non smooth extremals detected in \cite{AB} and \cite{AB2}, called \emph{broken extremals}, given a non linear affine control system (\ref{1,1}).\\

Let us briefly recall the conditions that we need in a neighbourhood $O_{\bar{\lambda}}$ of $\bar\lambda\in\Lambda\subseteq T^* M$ in order to have and study broken extremals.
\begin{notation} We denote $h_i(\lambda):=\left\langle \lambda,f_i(q) \right\rangle $ $h_{ij}(\lambda):=\left\langle \lambda,[f_{i},f_j](q) \right\rangle $, $\lambda\in T_q M$ and $i,j\in\{0,\hdots,k\}$.\\
Moreover, given $\bar{\lambda}\in \Lambda$, we call $h_{ij}=h_{ij}(\bar{\lambda})$, then $H_{0I}(\lambda)=(h_{0i})_{i=1,\hdots,k}$ and $H_{IJ}(\lambda)=(h_{ij})_{i,j=1,\hdots,k}$
\end{notation}
Given a $n$-dimensional manifold $M$, let us consider the system (\ref{1,1}). From the Pontryagin maximum principle, out of the singular locus $\Lambda=\{\lambda\in T^*M | h_1(\lambda)=\hdots=h_k(\lambda)=0\}$, extremals satisfy the Hamiltonian system defined by
$$H(\lambda)=h_0(\lambda)+\sqrt{h^2_1(\lambda)+\hdots+h^2_k(\lambda).}
$$
In \cite{AB2} we proved that if at $\bar\lambda\in\Lambda$ it is satisfied the condition
\begin{equation}
\label{p265}H_{0I}\notin H_{IJ}\overline{B}^k
\end{equation}
there exist a unique extremal that passes through $\bar\lambda$, moreover in its neighbourhood $O_{\bar\lambda}$ the continuous flow of extremals is defined. This flow is not locally Lipschitz in general.\\

Denoting $\bar{q}=\pi(\bar{\lambda})$, the projection of $\bar\lambda$ in $M$, and $\mathcal{F}=\{f_1,\hdots,f_k\}$, we prove the sufficient optimality of the normal broken extremal, passing through $\bar{\lambda}\in\Lambda$, if
$$\bar{\lambda}\perp\mathrm{Lie}_{\bar{q}}\mathcal{F},\quad h_0(\bar{\lambda})>0$$
and either $\mathrm{rank}\left\lbrace \mathrm{Lie}_{\bar{q}}\mathcal{F} \right\rbrace = n-1$, or $\mathrm{rank}\left\lbrace \mathrm{Lie}_{q}\mathcal{F} \right\rbrace=\mathrm{rank}\left\lbrace \mathrm{Lie}_{\bar{q}}\mathcal{F} \right\rbrace< n-1$ for all $q$ from a neighbourhood of $\bar{q}$ in $M$ (see Theorem \ref{teosuf1}). Moreover, if $n=3$ $k=2$ we prove the optimality for a normal broken extremal if $f_1,f_2$ form a contact distribution in a neighbourhood of $\bar{q}$ (see Theorem \ref{teosuf2}).\\
 We use a method described by Agrachev and Sachkov in their book \cite{A}. It is a geometrical elaboration of the classical fields of extremals theory, it proves optimality only for normal extremals, assuming the Hamiltonian smooth. We extended this method in the Lipschitzian submanifold, with constructions \emph{ad hoc}.\\
 
 We also prove optimality of normal (or abnormal) broken extremals for $n>2$ $k=2$ and 
\begin{equation}
\label{77t}
\bar\lambda\perp\mathrm{span}\{f_1(\bar{q}),f_2(\bar{q}),[f_1,f_2](\bar{q})\}
\end{equation}
in just that point (see Theorem \ref{teosuf3}). This result is given by direct estimates with time-rescaling.\\

In the thesis \cite{BC}, we present the computations of this method with direct estimates in the general (possible abnormal) case, if (\ref{77t}) does not hold. It may be useful to answer further questions. 

\section{Preliminaries} \label{sec:preliminaries.section}

In this section we recall some basic definitions in Geometric Control Theory. For a more detailed introduction, see \cite{A}.
\begin{definition}
Given a $n$-dimensional manifold $M$, we call $\mathrm{Vec}(M)$ the \emph{set of smooth vector fields} on $M$: $f\in \mathrm{Vec}(M)$ if and only if $f$ is a smooth map with respect to $q\in M$ taking value in the tangent bundle, $$f:M\longrightarrow TM,
$$ such that if $q\in M$ then $f(q)\in T_q M$.\\
Each vector field defines a \emph{dynamical system} $$\dot{q}=f(q),$$
i. e. for each initial point $q_0\in M$ it admits a solution $q(t,q_0)$ on an opportune time interval $I$, such that $q(0,q_0)=q_0$ and
$$\frac{d}{dt}q(t)=f(q(t)),\quad \mathrm{a.}\,\mathrm{e}. \,t\in I .
$$
\end{definition}

\begin{definition}
$f\in \mathrm{Vec}(M)$ is a \emph{complete vector field} if , for each initial point $q_0\in M$, the solution $q(t,q_0)$ of the \emph{dynamical system} $\dot{q}=f(q)$ is defined for every $t\in \mathbb{R}$. If $f\in \mathrm{Vec}(M)$ has a compact support, it is a complete vector field.
\end{definition}

In our local study, we may assume without lack of generality that all vector fields under consideration are complete.

\begin{definition}
A \emph{control system} in $M$ is a family of dynamical systems
$$\dot{q}=f_u(q), \quad \mathrm{with}\,\,q\in M,\, \{f_u\}_{u\in U}\subseteq \mathrm{Vec}(M),
$$
parametrized by $u\in U\subseteq\mathbb{R}^k$, called \emph{space of control parameters}.\\
Instead of constant values $u\in U$, we are going to consider $L^\infty$ time depending functions taking values in $U$. Thus, we call $\mathcal{U}=\{u:I\rightarrow U,\,u\in L^\infty\}$ the \emph{set of admissible controls} and study the following control system
\begin{equation}
\label{control.system}
\dot{q}=f_u(q), \quad \mathrm{with}\,\,q\in M,\, u\in \mathcal{U}.
\end{equation}
\end{definition}
With the following theorem we want to show that, choosing an admissible control, it is guaranteed the locally existence and uniqueness of the solution of a control system for every initial point.
\begin{theorem}
Fixed an admissible control $u\in\mathcal{U}$, (\ref{control.system}) is a non-autonomous ordinary differential equation, where the right-hand side is smooth with respect to $q$, and measurable essentially bounded with respect to $t$, then, for each $q_0\in M$, there exists a local unique solution $q_u(t,q_0)$ such that $q_u(0,q_0)=q_0$ and it is Lipschitzian with respect to $t$.
\end{theorem}
\begin{definition}
\label{007}
We denote
$$A_{q_0}=\{q_u(t,q_0) : t\geq 0,u\in\mathcal{U} \}
$$
the \emph{attainable set} from $q_0$. \\
We will write $q_u(t)=q_u(t,q_0)$ if we do not need to stress  that the initial position is $q_0$.
\end{definition}
\begin{definition}
An\emph{ affine control system} is a control system of the following form
\begin{equation}
\label{affine.control.system}
\dot{q}=f_0(q)+\sum^k_{i=1}u_if_i(q), \quad q\in M
\end{equation}
where $f_0,\ldots, f_k$ $\in \mathrm{Vec}(M)$ and $(u_1,\ldots, u_k)\in \mathcal{U}$, taking values in the set $U\subseteq \mathbb{R}^k$. \\
The uncontrollable term $f_0$ is called \emph{drift}. \\
\end{definition}
\subsection{Time-optimal problem}
\begin{definition}
Given the control system (\ref{control.system}), $q_0\in M$ and $q_1\in A_{q_0}$, the \emph{time-optimal problem} consists in minimizing the time of motion from $q_0$ to $q_1$ via admissible trajectories:
\begin{equation}
\label{time-optimal.problem}\left\lbrace \begin{array}{ll}
\dot{q}=f_u(q)&u\in \mathcal{U}\\
q_u(0,q_0)=q_0&\\
q_u(t_1,q_0)=q_1&\\
t_1 \rightarrow \min&
\end{array} \right.
\end{equation}
We call these minimizer trajectories \emph{time-optimal trajectories}, and \emph{time-optimal controls} the corresponding controls.
\end{definition}
\subsubsection{Existence of time-optimal trajectories}
Classical Filippov's Theorem (See \cite{A}) guarantees the existence of a time-optimal control for the affine control system if $U$ is a convex compact and $q_0$ is sufficiently close to $q_1$.
\subsection{First and second order necessary optimality condition}
Now we are going to introduce basic notions about Lie brackets, Hamiltonian systems and Poisson brackets, so that we present the first and second order necessary conditions of optimality: Pontryagin Maximum Principle, and Goh condition.
\begin{definition}
Let $f,g\in \mathrm{Vec}(M)$, we define their \emph{Lie brackets} the following vector field
$$[f,g](q)=\frac{1}{2}\left.\frac{\partial^2}{\partial t^2}\right|_{t=0}e^{-t g}\circ e^{-t f}\circ e^{t g}\circ e^{t f}(q), \quad  \forall q\in M.
$$
where $e^{-t f}$ is the flow defined by $-f$.\\
\begin{center}
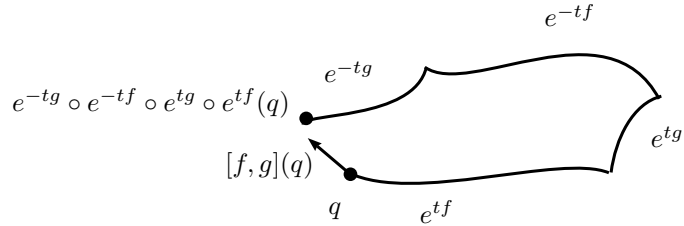

\psscalebox{1.0 1.0} 
{
\begin{pspicture}(0,-1.4669921)(6.62,1.4669921)
\psbezier[linecolor=black, linewidth=0.04](1.66,-0.73105466)(2.52,-1.1510547)(4.44,-0.47105467)(5.08,-0.7310546875)
\psbezier[linecolor=black, linewidth=0.04](5.1096,-0.72670466)(5.165974,-0.28143632)(5.3937507,0.19085746)(5.7704,0.2845953035607818)
\psbezier[linecolor=black, linewidth=0.04](5.74,0.24894531)(5.06,1.5289453)(3.26,0.24894531)(2.66,0.6689453125)
\psbezier[linecolor=black, linewidth=0.04](2.68,0.6689453)(2.5118368,-0.043554686)(1.0338775,0.071445316)(1.06,-0.0910546875)
\psline[linecolor=black, linewidth=0.04, arrowsize=0.05291667cm 2.0,arrowlength=1.4,arrowinset=0.12]{->}(1.68,-0.7510547)(1.12,-0.2710547)
\psdots[linecolor=black, dotsize=0.18](1.68,-0.7510547)
\psdots[linecolor=black, dotsize=0.18](1.1,-0.01)
\rput[bl](2.56,-1.3910546){$e^{tf}$}
\rput[bl](5.6,-0.3910547){$e^{tg}$}
\rput[bl](4.2,1.2089453){$e^{-tf}$}
\rput[bl](1.3,0.46894532){$e^{-tg}$}
\rput[bl](1.36,-1.3510547){$q$}
\rput[bl](0.0,-0.8110547){$[f,g](q)$}
\rput[bl](-2.8,0.0){$e^{-t g}\circ e^{-t f}\circ e^{t g}\circ e^{t f}(q)$}
\end{pspicture}
}
\captionof{figure}{Lie Bracket}
\end{center}
\end{definition}
\begin{definition}
A \emph{Hamiltonian} is a smooth function on the cotangent bundle
$$h\in C^\infty(T^*M).
$$
The \emph{Hamiltonian vector field} is the vector field associated with $h$ via the canonical symplectic form $\sigma$
$$\sigma_\lambda (\cdot ,\overrightarrow{h})=d_\lambda h.
$$
We denote
$$\dot{\lambda}=\overrightarrow{h}(\lambda), \quad  \lambda \in T^*M,
$$
the \emph{Hamiltonian system}, which corresponds to $h$.\\
Let $(x_1,\ldots,x_n)$ be local coordinates in $M$ and $(\xi_1,\ldots,\xi_n,x_1,\ldots,x_n)$ induced coordinates in $T^*M,\ \lambda=\sum_{i=1}^n\xi_idx_i$. The \emph{symplectic form} has expression $\sigma=\sum^n_{i=1}d\xi_i\wedge dx_i$. Thus, in canonical coordinates, the Hamiltonian vector field has the following form
$$\overrightarrow{h}=\sum^n_{i=1}\left( \frac{\partial h}{\partial \xi_i}\frac{\partial}{\partial x_i} -\frac{\partial h}{\partial x_i}\frac{\partial}{\partial \xi_i}\right).
$$
Therefore, in canonical coordinates, it is
$$\left\lbrace
\begin{array}{l}
\dot{x}_i=\frac{\partial h}{\partial \xi_i}\\
\dot{\xi_i}=-\frac{\partial h}{\partial x_i}
\end{array}
\right.
$$
for $i=1,\ldots,n$.
\end{definition}
\begin{definition}
The \emph{Poisson brackets} $\{a,b\}\in \mathcal{C}^\infty(T^*M)$ of two Hamiltonians $a,b\in \mathcal{C}^\infty(T^*M)$ are defined as follows: $\{a,b\}=\sigma(\vec a,\vec b)$; the coordinate expression is:
$$\{a,b\}=\sum_{k=1}^n\left( \frac{\partial a}{\partial \xi_k}\frac{\partial b}{\partial x_k}-\frac{\partial a}{\partial x_k}\frac{\partial b}{\partial \xi_k}\right).
$$
\end{definition}
\begin{remark}
\label{poisson,lie}
Let us recall that, given $g_1$ and $g_2$ vector fields in $M$, considering the Hamiltonians $a_1(\xi,x)=\left\langle \xi, g_1(x)\right\rangle $ and $a_2(\xi,x)=\left\langle \xi, g_2(x)\right\rangle $, it holds
$$\{a_1,a_2\}(\xi,x)=\left\langle \xi, [g_1,g_2](x)\right\rangle.
$$
\end{remark}
\begin{remark}
\label{derivPoiss}
Given a smooth function $\Phi$ in $\mathcal{C}^\infty(T^*M)$, and $\lambda(t)$ solution of the Hamiltonian system $\dot{\lambda}=\overrightarrow{h}(\lambda)$, the derivative of $\Phi(\lambda(t))$ with respect to $t$ is the following
$$\frac{d}{dt}\Phi(\lambda(t))=\{h,\Phi\}(\lambda(t)).
$$
\end{remark}
\subsubsection{Pontryagin Maximum Principle}
\begin{theorem}[Pontryagin Maximum Principle - time-optimal problem]
Let an admissible control $\tilde{u}$, defined in the interval $t\in [0,\tau_1 ]$, be time-optimal for the system (\ref{control.system}), and let the Hamiltonian associated with this control system be the action on $f_u(q)\in T^*_q M$ of a covector $\lambda\in T^*_q M$: $$\mathcal{H}_u(\lambda)=\left\langle \lambda,f_u(q)\right\rangle . $$
Then there exists $\lambda(t)\in T_{q_{\tilde{u}}(t)}^*M$, for $t\in [0,\tau_1 ]$, called \emph{extremal} never null and Lipschitzian, such that for almost all $t\in [0,\tau_1 ]$ the following conditions hold:
\begin{enumerate}
	\item $\dot{\lambda}(t)=\vec{\mathcal{H}}_{\tilde{u}}( \lambda(t))$
	\item $\mathcal{H}_{\tilde{u}}(\lambda(t))= \max_{u\in U} \mathcal{H}_u(\lambda(t))$ (Maximality condition)
\item $\mathcal{H}_{\tilde{u}}(\lambda(t))\geq0$.
\end{enumerate}
Given the canonical projection $\pi:TM\rightarrow M$, we denote $q(t)=\pi(\lambda(t))$ the \emph{extremal trajectory}.
\end{theorem}
\subsubsection{Goh condition}
\label{subsecgoh}
Finally, we present the Goh condition, on the singular arcs of the extremal trajectory, in which we do not have information from the maximality condition of the Pontryagin Maxinum Principle. We state the Goh condition only for affine control systems (\ref{affine.control.system}).
\begin{theorem}[\emph{Goh condition}]
\label{Goh.condition}
Let $\tilde q(t),\ t\in[0,t_1]$ be a time-optimal trajectory corresponding to a control $\tilde u$. If $\tilde u(t)\in\mathrm{int}U$ for any $t\in(\tau_1,\tau_2)$,
then there exist an extremal $\lambda(t)\in T_{q(t)}^*M$ such that
\begin{equation}
\label{cond.di.Goh}
\left\langle \lambda(t),[f_i,f_j](q(t))\right\rangle =0,\quad\ t\in(\tau_1,\tau_2),\ i,j=1,\ldots,m.
\end{equation}
\end{theorem}
\subsection{Broken extremals}
Let us define the broken extremals presenting some facts from paper \cite{AB2}.

We consider $n$-dimensional affine control system with a $k$-dimensional control:
\begin{equation}
\label{246}
\dot{q}=f_0(q)+\sum^k_{i=1}u_if_i(q),\quad q\in M, u\in\mathcal{U}
\end{equation}
where the space of control parameters is the $k$-dimensional closed unitary ball: $U=\{u\in\mathbb{R}^k : ||u||\leq 1\}$.\\
By the Pontryagin Maximum Principle, every time-optimal trajectory of our system has an extremal in the cotangent bundle $T^*M$ that satisfies a Hamiltonian system, given by the maximized Hamiltonian $H$, denoted by the maximality condition.
\begin{notation} Let us call $h_i(\lambda)=\left\langle \lambda,f_i(q) \right\rangle $, $f_{ij}(q)=[f_i,f_j](q)$, and $h_{ij}(\lambda)=\left\langle \lambda,f_{ij}(q) \right\rangle $, with $\lambda\in T^*_{q}M$ and $i,j\in\{0,1,\hdots, k\}$. 
\end{notation}
In this setting, we have the \emph{singular locus} $\Lambda \subseteq T^*M$ defined as follows
$$
\Lambda=\{\lambda\in T^*M : h_1(\lambda)=\hdots=h_k(\lambda)=0\},
$$
and the following proposition is an immediate corollary of the Pontryagin Maximum Principle.

\begin{proposition}
\label{776}
If an extremal $\lambda(t)$ of system (\ref{246}) does not intersect the singular locus $\Lambda$ at time  $t\in[0,t_1]$, then $\forall t\in [0,t_1]$
\begin{equation}
\label{time-opt.control}
\tilde{u}(t)=\left( \begin{array}{c}
\frac{h_1(\lambda(t))}{(h^2_1(\lambda(t))+\hdots+h^2_k(\lambda(t)))^{1/2}}\\
\vdots
\\
 \frac{h_k(\lambda(t))}{(h^2_1(\lambda(t))+\hdots+h^2_k(\lambda(t)))^{1/2}}
 \end{array}
 \right).
\end{equation}
Moreover, this extremal is a solutions of the Hamiltonian system defined by the Hamiltonian $\mathcal{H}(\lambda)=h_0(\lambda)+\sqrt{h^2_1(\lambda)+\hdots+h^2_k(\lambda)}$. Thus, it is smooth.
\end{proposition}
We will call \emph{bang arc} any smooth arc of a time-optimal trajectory $q(t)$, whose corresponding time-optimal control $\tilde{u}$ lies in the boundary of the space of control parameters: $\tilde{u}(t)\in \partial U$. Then an arc of a time-optimal trajectory, whose extremal is out of the singular locus, is a bang arc.\\
Thus,  every time-optimal trajectory, whose extremal lies out of the singular locus, is smooth.

However, one can observe that there is a singularity on a time-optimal trajectory if the corresponding extremal touches the singular locus $\Lambda$, and the optimal control has a discontinuity.\\
We call \emph{switching} a discontinuity of an optimal control.
\begin{definition}
We denote \emph{broken extremals} those extremals that pass through the singular locus at a point $\bar\lambda\in\Lambda$ and have a singularity in $\bar\lambda$, they are going to be defined in Theorem \ref{brokex}.
\end{definition}
\begin{notation}
Given $\bar\lambda\in\Lambda$, we denote the vector $H_{0I}=\{h_{0i}(\bar\lambda)\}_{i}\in \mathbb{R}^k$ and $k\times k$ matrix $H_{IJ}=\{h_{ij}(\bar\lambda)\}_{ij}$ with respect to $\lambda\in T^*M$.
\end{notation}
By the following Theorem we will see that, given condition (\ref{condnotin}) at $\bar\lambda\in\Lambda$, there exists a unique extremal through $\bar\lambda$ with an isolated singularity in $\bar\lambda$, namely a \emph{broken extremal}, and there exists a neighbourhood of $\bar\lambda$ $O_{\bar\lambda}$, where the flow of extremals is defined.
\begin{theorem}
\label{brokex}
If it holds
\begin{equation}
\label{condnotin}
H_{0I}\notin H_{IJ}\overline{B^{k}},
\end{equation}
where $B^k=\{u\in \mathbb{R}^k\,:\,||u||<1\}$, then there exists a neighborhood $O_{\bar\lambda}\subset T^*M$ such that for any $z\in O_{\bar\lambda}$ and $\hat{t}>0$ there exists a unique contained in $O_{\bar\lambda}$ extremal $t\mapsto\lambda(t,z)$ with the condition $\lambda(\hat{t},z)=z$. Moreover, $\lambda(t,z)$ continuously depends on $(t,z)$ and every extremal in $O_{\bar{\lambda}}$ that passes through the singular locus is piece-wise smooth with only one switching.\\
Besides that, if $u$ is the control corresponding to the extremal that passes through $\bar{\lambda}$, and $\bar{t}$ is its switching time, we have:
\begin{equation}
\label{jumpi}
u(\bar{t}\pm 0)=[\pm d\,\mathrm{Id}+H_{IJ}]^{-1}H_{0I},
\end{equation}
with $d\,>0$ unique, uni vocally defined by the system and $\bar{\lambda}$, such that
\begin{equation}
\label{iii}\left\langle [d^2\, \mathrm{Id}-H_{IJ}^2]^{-1}H_{0I},H_{0I} \right\rangle=1 .
\end{equation}
\end{theorem}
\begin{remark}
\label{lol8}
 In general, the flow of switching extremals through the singular locus is not locally Lipschitz with respect to the initial value. In \cite{AB} was found a simple counterexample that can be easily generalized to any $k<n$.
\end{remark}
\section{Sufficient optimality for normal extremals}
In this section we are going to see some cases in which we prove the sufficient optimality of normal extremals through the singular locus.\\
We used a method described by Agrachev and Sachkov in their book \cite{A}. It is a geometrical elaboration of the classical fields of extremals theory, it proves optimality only for normal extremals, assuming the Hamiltonian smooth. We extended this method with constructions \emph{ad hoc}.\\

Here we are going to show the generalized method that we can apply to the broken extremal defined by the system
\begin{equation}
\label{246h}
\dot{q}=f_0(q)+\sum^k_{i=1}u_if_i(q),\quad q\in M, u\in\mathcal{U}
\end{equation}
where the space of control parameters is the $k$-dimensional closed unitary ball $U=\{u\in\mathbb{R}^k : ||u||\leq 1\}$. In this setting extremals satisfies the Hamiltonian system denoted by the non smooth Hamiltonian 
$$H(\lambda)=h_0(\lambda)+\sqrt{h^2_1(\lambda)+\hdots+h^2_k(\lambda)}.
$$
Let us start considering only normal extremal that passes through the singular locus $\Lambda$. Hence, assuming that every $\lambda(t)$ must remain in the level set $H(\lambda(t))=1$, necessarily $H(\bar\lambda)=h_0(\bar\lambda)=1$.\\

Let us denote $s$ the tautological 1-form on $T^*M$, $s_\lambda=\lambda\circ \pi_*$, and its differential is the canonical symplectic structure in $T^*M$, d$s=\sigma$.
\begin{notation}
If $F:M\rightarrow N$ is a smooth mapping, we denote $F^*:\Lambda^k N\rightarrow \Lambda^k M$ the mapping of differential forms\\
if $\omega\in\Lambda^k N$, $(F^*\omega)_q(v_1,\hdots,v_k)=\omega_{F(q)}(F_*v_1,\hdots,F_*v_k)$, $q\in M$ $v_i\in T_qM$
\end{notation}
\begin{theorem}
\label{4658m}
Let $\tilde{\lambda}(t)$ be broken extremal passing through $\bar\lambda\in\Lambda$. If it is possible to define
\begin{itemize}
\item[(1)]A co-dimension one submanifold $N$ of $M$ such that the curve $\tilde{q}(t)\in\pi(\tilde{\lambda}(t))$ passes transversally through $N$ in both sides with $\tilde{q}(\bar{t})=\bar{q}=\pi(\bar\lambda)\in N$
\item[(2)]A section $\omega$ of bundle $T^*M_{|N}$ such that 
$$\omega:q\in N\rightarrow \omega_q\in T^*M,
$$
$H(\omega_q)=1$ and $\left\langle \omega_q,f_i(q) \right\rangle=0$ with $i\in\{1,\hdots,k\}$ for all $q\in N$; moreover $\omega_{|N}$ is a well defined differentiable 1-form of $N$ and $d\omega_{|N}=0$.
\end{itemize}
Then $\tilde{q}(t)$ is time-optimal at $\bar{q}$: there exists an interval $J=(t_1,t_2)$ with $\bar{t}\in J$, such that $\tilde{q}(t)$ with $t\in J$ realizes a strict minimum time among all admissible trajectories $q(t)$ such that $q(t_1)=\tilde{q}(t_1)$ and $q(\tau)=\tilde{q}(t_2)$ with $\tau>t_1$.
\end{theorem}
\begin{proof}
Given $N\subset M$ and $\omega: q\in N \rightarrow \omega_q\in T_{q}M$ with those hypothesis, let us consider 
$$\mathcal{N}=\{\omega_q\,:\,q\in N\}
$$
that is a submanifold in $T^*M$ such that $\mathcal{N}\subset\Lambda\cap H^{-1}(1)$, in particular $\bar\lambda\in \mathcal{N}$. 

From what we have proved in \cite{AB2}, if $\bar\lambda\in\Lambda$ satisfies condition (\ref{condnotin}), given $O_{\bar\lambda}$ a small enough neighbourhood of $\bar{\lambda}$, for all $\hat\lambda\in O_{\bar\lambda}\cap \Lambda$ there exists a unique broken extremal $\lambda_{\hat\lambda}(t)$ that passes through the singular locus at $\hat\lambda$. We assume $\lambda_{\hat\lambda}(\bar{t})=\hat\lambda$. Moreover, let us recall that out of $\Lambda$ each extremal satisfies the Hamiltonian system $\dot{\lambda}=\overrightarrow{H}(\lambda)$, with $H(\lambda)=h_0(\lambda)+\sqrt{h^2_1(\lambda)+\hdots+h^2_k(\lambda)}$.

Hence, we restrict $\mathcal{N}$ to those points close to $\bar\lambda$ and define the map 
$$\Phi : \mathcal{N}\times I \rightarrow T^*M
$$
where $I\subseteq \mathbb{R}$ is an interval with $\bar{t}\in I$, such that $\Phi(\hat\lambda,t)=\lambda_{\hat\lambda}(t)$.

From what we have explained in paper \cite{AB2}: given any $\hat\lambda\in \mathcal{N}$ $\lambda_{\hat\lambda}(t)$ is piece-wise smooth with respect to $t$ in the two sides where $t<\bar{t}$ or $t>\bar{t}$, and it is globally lipschitzian because the right and left limits of $\dot{\lambda}(t)$ as $t\rightarrow \bar{t}\pm 0$ are well defined.\\
Moreover, Theorem \ref{brokex} claims that at $O_{\bar{\lambda}}$ it is defined a continuous flow of extremals that is not locally Lipschitz. Nevertheless, considering just broken extremals passing through $\mathcal{N}$, the image of map $\Phi$ is a piece-wise smooth manifold composed by two smooth manifolds with boundary $\mathcal{N}$. Globally $\Phi(\mathcal{N}\times I)$ is Lipschitzian, because we have that $\frac{\partial \Phi}{\partial t}(\hat\lambda,t)_{|t\neq \bar{t}}=\overrightarrow{H}\left(\lambda_{\hat\lambda}(t))\right)$ for all $(\hat\lambda,t)\in \mathcal{N}\times \{I\setminus \{\bar{t}\}\}$ and the limits as $t\rightarrow \bar{t}\pm 0$ are explicitly defined(see \cite{AB2}).\\

Let us stress that, given $W$ a domain in $\mathcal{N}\times I$ such that $(\bar{\lambda},0)\in W$, the map
$$\pi\circ \Phi_{|W}: W \rightarrow M
$$
is a Lipschitzian (even piece-wise smooth) homeomorphism of $W$ into a domain in $M$, by construction. This is because we assume that $\tilde{q}(t)$ passes transversally through $N$ in both sides.

As a consequence, we have that $\Phi$ is a piece-wise smooth immersion since $\pi\circ \Phi_{|W}$ is immersion.\\

Now, we need to prove a technical fact: $\Phi^*s$ is an exact form.\\

It is a closed form, because
$$d(\Phi^*s)_{|(\hat\lambda,t)}=\Phi^*\sigma_{|(\hat\lambda,t)}=\sigma_{|\Phi(\hat\lambda,t)}\quad \forall (\hat\lambda,t)\in \mathcal{N}\times \mathbb{R}
$$
by the properties of the exterior derivative, and
$$\sigma_{|\Phi(\hat\lambda,t)}=\sigma_{|\hat\lambda}=ds_{|\hat\lambda}=d(\omega\circ\pi)_{|\hat\lambda}=0 \quad \forall (\hat\lambda,t)\in \mathcal{N}\times \mathbb{R}
$$
because of the properties of form $\sigma$ and by definition of $\mathcal{N}$.

On the other hand, it is exact because, given any closed curve
$$\gamma : \tau \mapsto (\lambda_0(\tau),t(\tau))\in\mathcal{N}\times\mathbb{R},
$$
one can see that
$$\int_{\gamma}\Phi^*s=0.
$$
We have 
$$\int_{\gamma}\Phi^*s=\int_{\Phi(\gamma)}s
$$
$\Phi(\gamma)$ is homeomorphic to 
$$\gamma_0 : \tau \mapsto \lambda_0(\tau)\in \mathcal{N},
$$
then, by the Lipschitzian version of Stokes Theorem \cite{L} and the definition of $\mathcal{N}$
$$\int_{\Phi(\gamma)}s=\int_{\gamma_0}s=\int_{\gamma_0}\omega\circ \pi =0.
$$
\\
\bigskip
Finally, we prove the thesis of the theorem.\\
Let us call $\mathcal{N}_{W}=\Phi(W)\subset T^*M$ such that $\pi:\mathcal{N}_{W}\rightarrow \pi(\mathcal{N}_{W})$ is a Lipschitzian (even piece-wise smooth) homeomorphism and $s_{|\mathcal{N}_{W}}$ is an exact form.\\
Given $\tilde{q}(t)=\pi(\tilde{\lambda}(t))$ with $t\in(t_1,t_2)$ such that $t_1<0<t_2$ and $\tilde{q}(0)=\bar{q}=\pi(\bar\lambda)$, let us consider $q(t)$ with $t\in (t_1,\tau)$ an admissible trajectory generated by a control $u(t)$ and contained in $\pi(\mathcal{N}_{W})$, with the boundary conditions $q(t_1)=\tilde{q}(t_1)$ and $q(\tau)=\tilde{q}(t_2)$. Then, by the map $\pi_{|\mathcal{N}_{W}}$, there exists a curve $\lambda(\cdot):t\mapsto \lambda(t)$ in $\mathcal{N}_{W}$ such that $\lambda(t_1)=\tilde{\lambda}(t_1)$ $\lambda(\tau)=\tilde{\lambda}(t_2)$ and $q(t)=\pi(\lambda(t))$ for all $t\in (t_1,\tau)$.\\
Since $\int_{\lambda(\cdot)}s=\int_{\tilde\lambda(\cdot)}s$ and $H(\lambda(t))=\max_{u\in U}\left\langle \lambda(t),f_0(q(t))+\sum^k_{i=1}u_i(t)f_i(q(t))\right\rangle =1$ we have
$$\int_{\tilde{\lambda}}s = \int^{t_1}_{0}\left\langle \tilde{\lambda}_t,\dot{\tilde{q}}(t) \right\rangle dt= \int^{t_1}_{0}\underbrace{\left\langle \tilde{\lambda}_t,f_0(\tilde{q}(t))+\sum^k_{i=1}\tilde{u}_i(t)f_i(\tilde{q}(t)) \right\rangle}_{=H(\tilde{\lambda}(t))=1} dt=t_1.
$$
On the other hand,
$$\int_{\lambda(\cdot)}s = \int^{\tau}_{0}\left\langle \lambda(t),\dot{q}(t) \right\rangle dt= \int^{\tau}_{0}\underbrace{\left\langle \lambda(t),f_0(q(t))+\sum^k_{i=1}u_i(t)f_i(q(t))  \right\rangle}_{\leq 1} dt\leq \tau
$$
Moreover, the inequality is strict if the curve $t\mapsto \lambda(t)$ is not a solution of the equation $\dot{\lambda}=\overrightarrow{H}(\lambda)$, namely is does not coincide with $\tilde{\lambda}(t)$.\\

Therefore, we have proved that $\tilde{q}(t)$ is locally time-optimal in the switching point $\bar{q}$. Actually, it is globally optimal.\\
It is optimal with respect to the whole trajectory. Indeed, it will spend strictly more time going out side the neighbourhood.
\begin{center}
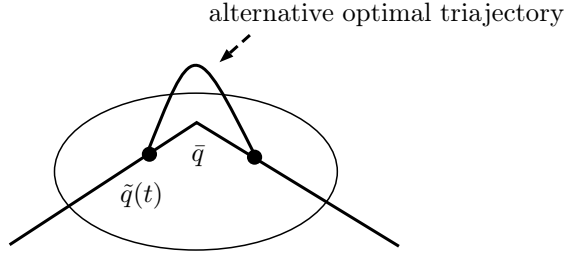

\psscalebox{1.0 1.0} 
{
\begin{pspicture}(0,-1.6276367)(7.7509837,1.6276367)
\psellipse[linecolor=black, linewidth=0.02, dimen=outer](2.4609835,-0.5776367)(1.87,1.05)
\psline[linecolor=black, linewidth=0.04](0.010983582,-1.5476367)(2.4709835,0.07236328)(5.1309834,-1.5676367)
\psbezier[linecolor=black, linewidth=0.04](1.7909836,-0.40763673)(2.4309835,1.2323632)(2.4309835,1.2523633)(3.2509835,-0.38763671875)
\psdots[linecolor=black, dotsize=0.2](1.8509836,-0.34763673)
\psdots[linecolor=black, dotsize=0.2](3.2309835,-0.38763672)
\rput[bl](2.3709836,-0.5076367){$\bar{q}$}
\rput[bl](1.4309835,-1.0676367){$\tilde{q}(t)$}
\rput[bl](2.6309836,1.3523632){alternative optimal triajectory}
\psline[linecolor=black, linewidth=0.04, linestyle=dashed, dash=0.17638889cm 0.10583334cm, arrowsize=0.05291667cm 2.0,arrowlength=1.4,arrowinset=0.0]{->}(3.1709836,1.2323632)(2.7709837,0.87236327)
\end{pspicture}
}
\captionof{figure}{Global optimality of $\tilde{q}(t)$.}
\end{center}
\end{proof}
\begin{remark}
\label{plpl}
On the other hand, if we study the problem with a smooth Hamiltonian $H$ and $\overrightarrow{H}$ complete, it is enough give a Lagrangian $\mathcal{N}$ such that
$$\mathcal{N}=\left\lbrace \mathrm{d}_q a\,|\,q\in M\right\rbrace,
$$
with $a\in \mathcal{C}^\infty(M)$ any arbitrary smooth function.\\
As a consequence $\omega=\mathrm{d}a$.
\end{remark}
Now, let us present two cases in which we found such a Lagrangian submanifold $\mathcal{N}$.\\
This method proves the optimality of those extremals that pass through $\Lambda$.
\begin{theorem}
\label{teosuf1}
Given an affine control system (\ref{246h}) with $f_1,\hdots,f_k$ analytic fields, let $\bar\lambda\in\Lambda$ be a singular point such that it holds (\ref{p265}) and $H(\bar\lambda)=1$.\\
In this setting let us consider the normal broken extremal $\tilde{\lambda}(t)$ through $\bar\lambda$, such that $\tilde{\lambda}(0)=\bar{\lambda}$.\\
We denote $\mathcal{F}=\{f_1,\hdots,f_k\}$ the family of controllable vector fields. If
$$\bar{\lambda}\perp\mathrm{Lie}_{\bar{q}}\mathcal{F}, \quad h_0(\bar{\lambda})>0$$
 and either $\mathrm{rank}\left\lbrace \mathrm{Lie}_{\bar{q}}\mathcal{F} \right\rbrace = n-1$, or $\mathrm{rank}\left\lbrace \mathrm{Lie}_{q}\mathcal{F} \right\rbrace=\mathrm{rank}\left\lbrace \mathrm{Lie}_{\bar{q}}\mathcal{F} \right\rbrace< n-1$ for all $q$ from a neighbourhood $O_{\bar{q}}$ of $\bar{q}$ in $M$, then $\tilde{q}(t)=\pi(\tilde\lambda(t))$ is locally time-optimal among all admissible trajectory in $O_{\bar{q}}$ with the same boundary conditions.
\end{theorem}
\begin{proof}
As we discussed previously, it is enough find an opportune Lagrangian submanifold $\mathcal{N}$ with the said conditions.\\
Let us consider $\mathcal{F}$ and the distribution $\mathrm{Lie}_q\mathcal{F}$, that, by definition, is closed with respect to the Lie brackets.

If $\mathrm{rank}\{\mathrm{Lie}_{\bar{q}}\mathcal{F}\}=n-1$, we will denote $N$ the orbit $\mathcal{O}_{\bar{q}}$ of distribution $\mathrm{Lie}\mathcal{F}$ at point $\bar{q}$ that is a $n-1$ dimensional submanifold of $M$ by Nagano Theorem (see \cite{A}).

Otherwise, if $\mathrm{rank}\left\lbrace \mathrm{Lie}_{q}\mathcal{F} \right\rbrace =\mathrm{rank}\left\lbrace \mathrm{Lie}_{\bar{q}}\mathcal{F} \right\rbrace =m<n-1$ $\forall q\in O_{\bar{q}}$, then, by Frobenius Theorem (see \cite{A}), it is defined a fibration in $O_{\bar{q}}$ give by the $m$ dimensional submanifold $N'$ of $M$ and other $n-m$ components. By construction, one can define the codimension $1$ submanifold $N$, such that  such that $N'\subset N$, $\mathcal{F}\subseteq T_q N$ $\forall q\in O_{\bar{q}}$ and $f_0(\bar{q})\notin T_{\bar{q}}N$, and the said curve $\tilde{q}(t)$ will cross transversally $N$ at $\bar{q}$.

Moreover, let us define $\omega$ the $1$-form that annihilates $T_q N$ such that $\omega(f_0)_{|q}=1$, for all $q\in N$. By construction, $\omega$ satisfies $d\omega_{|N}=0$, and denoting
$$\mathcal{N}=\{\omega_{|q}\,|\,q\in N\},
$$
it holds $\mathcal{N}\subseteq \Lambda \cap H^{-1}(1)$.\\
All these facts imply the thesis.
\end{proof}
\begin{remark} Let us notice that, in the setting of Theorem \ref{teosuf1}, the corresponding smooth function $a$, denoted in Remark \ref{plpl}, is such that $a_{|N}=0$ and $\mathrm{d}a=\omega$.
\end{remark}
Before presenting the next result let us define the Reeb vector field.
\begin{definition}
In a $3$-dimensional manifold $M$ let us consider a \emph{contact 1-form} $\omega\in\Lambda^1(M)$ such that $\omega\wedge d\omega\neq 0$ in never vanishing. \\
The \emph{Reeb vector field} $\xi\in\mathrm{Vec}(M)$ is the unique element of the (one-dimensional) kernel of $d\omega$ such that $\omega(\xi)=1$.
\end{definition}
\begin{theorem}
\label{teosuf2}
Given an affine control system (\ref{246h}) with $n=3$ and $k=2$ and  with $f_1,f_2$ analytic fields, let $\bar\lambda\in\Lambda$ be a singular point such that it holds (\ref{p265}) and $H(\bar\lambda)=1$.\\
In this setting let us consider the normal extremal $\tilde{\lambda}(t)$ through $\bar\lambda$, such that $\tilde{\lambda}(0)=\bar{\lambda}$.\\
If the distribution
$$\Delta=\mathrm{span}\{f_1,f_2\}
$$
is contact in $\bar{q}=\pi(\bar\lambda)$, then $\tilde{q}(t)=\pi(\tilde\lambda(t))$ is locally time-optimal among all admissible trajectory in a neighbourhood $O_{\bar{q}}$ of $\bar{q}$ with the same boundary conditions.
\end{theorem}
\begin{proof}
Since $\Delta$ is a contact distribution in a neighbourhood $O_{\bar{q}}$ of $\bar{q}$, there exists $\omega\in\Lambda^1(M)$ a 1-form such that $\Delta=\mathrm{ker}(\omega)$ and $
\omega\wedge d \omega\neq 0$, moreover we can assume $\omega_q(f_0(q))\equiv 1$ $\forall q\in O_{\bar{q}}$.\\
We can define $\xi\in\mathrm{Vec}(M)$ the Reeb field such that $\left\langle \xi \right\rangle = \mathrm{ker}(d\omega)$.\\
We construct a co-dimension $1$ submanifold $N$ in the following way.\\
Given the control $\tilde{u}$ corresponding to the extremal trajectory $\tilde{q}(t)=\pi(\tilde{\lambda}_t)$, let us denote $f_-(\bar{q})$ and $f_+(\bar{q})$ at point $\bar{q}$
$$\left\lbrace 
\begin{array}{l}
f_-(\bar{q})=\tilde{u}_1(0^-)f_1(\bar{q})+\tilde{u}_2(0^-)f_2(\bar{q})\\
f_+(\bar{q})=\tilde{u}_1(0^+)f_1(\bar{q})+\tilde{u}_2(0^+)f_2(\bar{q}).
\end{array}
\right. 
$$
Let us give any integral curve $\hat{\gamma}$ whose velocities belong to the distribution $\mathrm{span}\{f_1(q),f_2(q)\}$, with $q\in O_{\bar{q}}$, as follows such that $f_-(\bar{q})$ and $f_+(\bar{q})$ appear in the same side.
\begin{center}
\psscalebox{1.0 1.0} 
{
\begin{pspicture}(0,-2.4919086)(8.2,2.4919086)
\psline[linecolor=black, linewidth=0.04](4.2,2.4720142)(7.86,2.092014)(5.86,-2.4679859)(1.76,-1.3279859)(4.18,2.4720142)
\psbezier[linecolor=black, linewidth=0.062](3.08,-0.28798583)(4.38,-0.85581195)(5.2,1.4841881)(6.5,1.73201416015625)
\psline[linecolor=black, linewidth=0.04, arrowsize=0.05291667cm 2.0,arrowlength=1.4,arrowinset=0.0]{->}(4.88,0.53201413)(6.32,0.87201416)
\psline[linecolor=black, linewidth=0.04, arrowsize=0.05291667cm 2.0,arrowlength=1.4,arrowinset=0.0]{->}(4.92,0.53201413)(5.18,-0.6479858)
\rput[bl](0.0,1.2720141){$\mathrm{span}\{f_1(\bar{q}),f_2(\bar{q})\}$}
\rput[bl](5.78,0.27201417){$f_-(\bar{q})$}
\rput[bl](5.42,-0.42798585){$f_+(\bar{q})$}
\psline[linecolor=black, linewidth=0.04, linestyle=dotted, dotsep=0.10583334cm, arrowsize=0.05291667cm 2.0,arrowlength=1.4,arrowinset=0.0]{->}(3.9,0.19201416)(4.62,0.27201417)(4.7,-0.92798585)
\psline[linecolor=black, linewidth=0.04, linestyle=dotted, dotsep=0.10583334cm, arrowsize=0.05291667cm 2.0,arrowlength=1.4,arrowinset=0.0]{->}(3.52,0.01201416)(4.02,-0.24798584)(4.04,-1.0079858)
\psline[linecolor=black, linewidth=0.04, linestyle=dotted, dotsep=0.10583334cm, arrowsize=0.05291667cm 2.0,arrowlength=1.4,arrowinset=0.0]{->}(4.34,0.9120142)(5.36,1.0120142)(5.58,0.07201416)
\psline[linecolor=black, linewidth=0.04, linestyle=dotted, dotsep=0.10583334cm, arrowsize=0.05291667cm 2.0,arrowlength=1.4,arrowinset=0.0]{->}(4.9,1.6520141)(5.98,1.5520141)(6.06,1.1120142)
\psdots[linecolor=black, dotsize=0.16](4.92,0.49201417)
\rput[bl](4.42,0.51201415){$\bar{q}$}
\end{pspicture}
}
\captionof{figure}{Curve in $\mathrm{span}\{f_1,f_2\}$}
\end{center}
Then we apply the flow generated by the Reeb field $\xi$. We denote this surface $N$.\\
Therefore we denote $$\mathcal{N}=\{\omega_q\in T^*M\,|\,q\in N \}.$$
Let us stress the fact that we chose the curve in $\mathrm{span}\{f_1,f_2\}$, as it is described at Figure 2, because we need to assume that $\tilde{q}(t)$ passes transversally through $N$.\\
This construction implies the thesis.
\end{proof}
\begin{remark} Let us notice that, in the setting of Theorem \ref{teosuf2}, the corresponding smooth function $a$, denoted in Remark \ref{plpl}, is the time-function along the Reeb curves such that $a(\hat{\gamma})\equiv 0$.\\
Indeed, given $\gamma(t)$ a curve along the Reeb flow with $\gamma(0)\in \hat{\gamma}$, we have
$$a(\gamma(t))=\underbrace{a(\gamma(0))}_{=0}+ \int^t_0\frac{d}{dt}a(\gamma(\tau)) d\tau=\int^t_0\underbrace{\left\langle d_{\gamma(\tau)}a,\dot{\gamma}(\tau) \right\rangle}_{\left\langle \omega_{\gamma(\tau)}, \xi(\gamma(\tau))\right\rangle=1 }  d\tau=t.
$$
\end{remark}

\section[Sufficient optimality, with 2D control]{Sufficient optimality, with 2-dimensional control}
In this new section we are going to present an alternative method to prove the sufficient optimality of extremals through the singular locus defined by systems of type (\ref{246h}) with 2-dimensional control.

At first we present how we reduce the problem and then the result that we were able to gave.
\subsection{How we reduce the problem}
Let us consider a control system of type (\ref{246h}) in the n-dimensional manifold $M$ with $k=2$. We assume $\bar{\lambda}\in\Lambda$ satisfying the condition 
$$H_{0I}\notin H_{IJ}\overline{B^k},
$$
namely there exists an extremal $\lambda(t)$ that passes through $\bar\lambda$, going through the singular locus.
Let us consider the perturbation of a trajectory $q(t)=\pi(\lambda(t))$ that is the projection of the extremal.\\
With opportune rotation of the system, we may assume that $\bar\lambda=\lambda(0)$ and the trajectory $q(t)$ satisfies the following system with constant piece-wise control:
$$\left\lbrace
\begin{array}{l}
\dot{q}=f_-(q):=f_0(q)+\cos(\hat\theta)f_1(q)-\sin(\hat\theta)f_2(q), \quad t<0 \\
\dot{q}=f_+(q):=f_0(q)+\cos(\hat\theta)f_1(q)+\sin(\hat\theta)f_2(q), \quad t>0\\
q(0)=\bar{q},
\end{array}
 \right.
$$
with $\hat{\theta}\in(0,\frac{\pi}{2})$.\\
As we saw in \cite{AB} and \cite{AB2}, it is possible to give explicitly the jump $u(\bar{t}\pm 0)$ of the control at the switching by equation 
$$
u(\bar{t}\pm 0)=\frac 1{r^2}\left( -h_{02}h_{12}\pm h_{01}(r^2-h^2_{12})^{\frac 12},h_{01}h_{12}\pm  h_{02}(r^2-h^2_{12})^{\frac 12}\right).
$$
In this setting we will have at $\bar\lambda$ $h_{01}=0$, $h_{02}>0$ and $h_{12}\leq 0$, and calling $\alpha:=\frac{|h_{12}|}{|h_{02}|}$ we have
$$(\cos(\hat{\theta}),\sin(\hat{\theta}))=\left(\alpha,  \sqrt{1-\alpha^2}  \right).
$$
In order to perturb the control with admissible controls, we denote 
$$g_{v}(q):=v_1f_1(q)+v_2f_2(q),$$
where $v_1$ and $v_2$ are time depending function such that
\begin{equation}
\label{t239}\left\lbrace 
\begin{array}{l}
\left|\left|\left(\alpha+v_1(t),-\sqrt{1-\alpha^2}+v_2(t)\right)\right|\right|\leq 1, \quad t<0\\
\\
\left|\left|\left(\alpha+v_1(t),\,\sqrt{1-\alpha^2}+v_2(t)\right)\right|\right|\leq 1, \quad t>0.
\end{array}
\right. 
\end{equation}
Defining
$$a:=\left(\begin{array}{c}
\alpha\\
\sqrt{1-\alpha^2}
\end{array}\right)
$$
the condition (\ref{t239}) becomes
\begin{equation}
\label{t329}\left\lbrace 
\begin{array}{l}
\left|\left(\begin{array}{c}
v_1\\
-v_2
\end{array}\right)\right|^2\leq -2\left(\begin{array}{c}
v_1\\
-v_2
\end{array}\right)\cdot a, \quad t<0\\
\left|\left(\begin{array}{c}
v_1\\
v_2
\end{array}\right)\right|^2\leq -2\left(\begin{array}{c}
v_1\\
v_2
\end{array}\right)\cdot a, \quad t>0.\\
\end{array}
\right. 
\end{equation}
Now, let us study the behaviour of the following path in the neighbourhood $O_{\bar{q}}$ of $\bar q$ at time $t\in [-\varepsilon,\varepsilon]$, with $\varepsilon> 0$ small, using the chronological calculus described in \cite{A} Chapter 2,
$$\bar q\,\circ F_{\varepsilon}(v)=\bar q\,\circ\,e^{(-\varepsilon-0) f_{-}}\,\circ\,\overrightarrow{\mathrm{exp}}\int^0_{-\varepsilon} f_- + g_v\,\mathrm{d} t\,\circ\,\overrightarrow{\mathrm{exp}}\int^\varepsilon_{0} f_+ + g_v\,\mathrm{d} t \,\circ\,e^{(0-\varepsilon) f_+}.
$$
\begin{claim}
\label{l682}
In order to prove the optimality of the switched curve among all perturbations, it is enough to prove the following:\\
\textbf{\emph{ Statement}}: \\
There exists $\bar{\varepsilon}>0$ such that $\forall v\neq 0$ and $\forall \varepsilon<\bar\varepsilon$ the functional $F_\varepsilon(v)\neq \mathrm{Id}$.
\end{claim}
Now, let us study deeply this functional $F_\varepsilon(v)$.\\
Thanks to the variational formula we simplify it in such a way
$$\begin{array}{l}
\bar q\,\circ F_{\varepsilon}(v)=\bar q\,\circ\,e^{-\varepsilon f_{-}}\,\circ\,e^{\varepsilon f_{-}}\,\circ\,\overrightarrow{\mathrm{exp}}\int^0_{-\varepsilon} e^{t\,\mathrm{ad}f_-}g_{v}\,\mathrm{d} t\,\circ\,\overrightarrow{\mathrm{exp}}\int^\varepsilon_{0} e^{t\mathrm{ad}f_+}g_v\,\mathrm{d} t \,\circ\,e^{\varepsilon f_+} \,\circ\,e^{-\varepsilon f_+}\\ \\
=\bar q\,\circ\,\overrightarrow{\mathrm{exp}}\int^0_{-\varepsilon} e^{t\,\mathrm{ad}f_-}g_{v}\,\mathrm{d} t\,\circ\,\overrightarrow{\mathrm{exp}}\int^\varepsilon_{0} e^{t\mathrm{ad}f_+}g_v\,\mathrm{d} t \\
\end{array}
$$
Rescaling the time in the integrals we have
$$
\begin{array}{l}
\bar q\,\circ F_{\varepsilon}(v)=\bar q\,\circ\,\overrightarrow{\mathrm{exp}}\int^0_{-1} \varepsilon e^{\varepsilon\,t\mathrm{ad}f_-}g_{v}\,\mathrm{d} t\,\circ\,\overrightarrow{\mathrm{exp}}\int^1_{0} \varepsilon e^{\varepsilon t\,\mathrm{ad}f_+}g_v\,\mathrm{d} t.
\end{array}
$$
Hence, we can rewrite it as follows
$$F_{\varepsilon}(v)=\overrightarrow{\mathrm{exp}}\int^1_{-1}\, V_t(\varepsilon)\,\mathrm{d} t
$$
where
$$\overrightarrow{\mathrm{exp}}\int^1_{-1} \, V_t(\varepsilon)\,\mathrm{d} t=\overrightarrow{\mathrm{exp}}\int^0_{-1}\varepsilon \, g^-_{\varepsilon  t}(v)\,\mathrm{d} t\, \circ \overrightarrow{\mathrm{exp}}\int^1_{0}\varepsilon g^+_{\varepsilon t}(v)\,\mathrm{d} t\,$$
and
 $$g^-_{\varepsilon  t}(v)=e^{\varepsilon t\, \mathrm{ad} \,f_{-}}g_{v}$$
$$g^+_{\varepsilon  t}(v)=e^{\varepsilon t \, \mathrm{ad} \,f_{+}}g_{v}.$$
\begin{notation}
We will use the following notation
$$F_{\varepsilon}(v)_{|[t\rightarrow 1]}=\overrightarrow{\mathrm{exp}}\int^1_{t} \, V_\tau(\varepsilon)\,\mathrm{d} \tau
$$
and
$$F_{\varepsilon}(v)_{|[1\rightarrow t]}=\overrightarrow{\mathrm{exp}}\int^t_{1} \, V_\tau(\varepsilon)\,\mathrm{d} \tau.
$$
\end{notation}
In order to verify what we state in Claim \ref{l682}, we are going to study the Taylor expansion of $F_\varepsilon(v)$
$$F_\varepsilon(v)=\mathrm{Id}+\partial_\varepsilon F_\varepsilon (v)_{|\varepsilon=0}\,\,\varepsilon+\frac{1}{2}\partial^2_\varepsilon F_{\varepsilon}(v)_{|\varepsilon=0}\,\,\varepsilon^2+O(\varepsilon^3)
$$
then the first derivative is
$$\partial_\varepsilon F_\varepsilon(v)=F_{\varepsilon}(v)\circ \int^1_{-1}F_\varepsilon(v)_{|[t\rightarrow 1]}\circ \partial_\varepsilon V_t(\varepsilon)\circ F_\varepsilon(v)_{|[1\rightarrow t]}\mathrm{d}t
$$
and the second
$$
\begin{array}{rcl}
\partial^2_\varepsilon F_\varepsilon(v)&=&\partial_\varepsilon F_\varepsilon(v)\circ \int^1_{-1}F_\varepsilon(v)_{|[t\rightarrow 1]}\circ \partial_\varepsilon V_t(\varepsilon)\circ F_\varepsilon(v)_{|[1\rightarrow t]}\mathrm{d}t+\\ \\
&&+F_{\varepsilon}(v)\circ \int^1_{-1}\partial_\varepsilon F_\varepsilon(v)_{|[t\rightarrow 1]}\circ \partial_\varepsilon V_t(\varepsilon)\circ F_\varepsilon(v)_{|[1\rightarrow t]}\mathrm{d}t+\\ \\
&&+F_{\varepsilon}(v)\circ \int^1_{-1}F_\varepsilon(v)_{|[t\rightarrow 1]}\circ \partial^2_\varepsilon V_t(\varepsilon)\circ F_\varepsilon(v)_{|[1\rightarrow t]}\mathrm{d}t+\\ \\
&&+F_{\varepsilon}(v)\circ \int^1_{-1}F_\varepsilon(v)_{|[t\rightarrow 1]}\circ \partial_\varepsilon V_t(\varepsilon)\circ \partial_\varepsilon F_\varepsilon(v)_{|[1\rightarrow t]}\mathrm{d}t.
\end{array}
$$
Since by construction
$$\begin{array}{lcl}
F_\varepsilon(v)_{|\varepsilon=0}=\mathrm{Id}&&\partial_\varepsilon V_t(\varepsilon)_{| \varepsilon=0}=g_v
\end{array}
$$
and
$$\int^1_{-1}\partial^2_\varepsilon V_t(\varepsilon)_{| \varepsilon=0}=\int^0_{-1}2t[f_-,g_v]\mathrm{d}t+ \int^1_{0}2t[f_+,g_v]\mathrm{d}t
$$
it holds
$$\partial_\varepsilon F_\varepsilon(v)_{|\epsilon=0}=\int^1_{-1}g_v \mathrm{d}t,
$$
and
$$
\begin{array}{r}\partial^2_\varepsilon F_\varepsilon (v)_{| \varepsilon =0}=\left[ \int^0_{-1}2t[f_-,g_v]\mathrm{d}t+ \int^1_{0}2t[f_+,g_v]\mathrm{d}t \right]+\\
\\+\int^1_{-1}\int^1_t[g_{v(\theta)},g_v]\mathrm{d}\theta\mathrm{d}t+\int^1_{-1}g_v\circ \int^1_{-1}g_v
\end{array}
$$
\begin{remark}
Given the Taylor expansion
$$F_\varepsilon(v)=\mathrm{Id}+\partial_\varepsilon F_\varepsilon (v)_{|\varepsilon=0}\,\,\varepsilon+\frac{1}{2}\partial^2_\varepsilon F_{\varepsilon}(v)_{|\varepsilon=0}\,\,\varepsilon^2+O(\varepsilon^3),
$$
if $\partial_\varepsilon F_\varepsilon (v)_{|\varepsilon=0}\,\,+\frac{1}{2}\varepsilon\,\partial^2_\varepsilon F_{\varepsilon}(v)_{|\varepsilon=0}\neq 0$ then the statement of Claim \ref{l682} is proved.
\end{remark}
\bigskip
Thus, we are interested in proving if the statement of Claim \ref{l682} can be proved even in the worst case. So, let us assume that
$$\partial_\varepsilon F_\varepsilon (v)_{|\varepsilon=0}\,\,+\frac{1}{2}\varepsilon\,\partial^2_\varepsilon F_{\varepsilon}(v)_{|\varepsilon=0} =0,$$
then 
$$\partial_\varepsilon F_\varepsilon (v)_{|\varepsilon=0}=\int_{-1}^1g_{v(t)}\mathrm{d} t\in O(\varepsilon),
$$
and 
$$\int^1_{-1}g_v\circ \int^1_{-1}g_v\in O(\varepsilon^2),
$$
and finally we rewrite the functional in the following way
$$\begin{array}{rcl}\frac{1}{\varepsilon}\left(F_{\varepsilon}(v)-\mathrm{Id}\right)&=&\int_{-1}^1g_{v(t)}\mathrm{d} t +\frac{1}{2}\varepsilon\left(\int^0_{-1} 2  t[f_-,g_{v }]\mathrm{d} t+\,\int^1_{0} 2t[f_+,g_{v}]\mathrm{d} t\right. \\
&&\\
&&\left. +\int^1_{-1}\int^{1}_t [g_{v(\tau)},g_{v(t)}]\mathrm{d} \tau\,\mathrm{d} t\, \right) +O(\varepsilon^2).\\
\end{array}
$$
At this point we calculate and study the scalar product $\left\langle \bar{\lambda}, \frac{1}{\varepsilon}\left(F_{\varepsilon}(v)-\mathrm{Id}\right) \right\rangle$, with $\bar\lambda\in \Lambda$, because if we show that it is strictly negative, the statement is proven and the projection of the extremal that we are analysing is optimal.\\

Thus, we have
\begin{equation}
\label{w4750}\begin{array}{rcl}
\left\langle \bar{\lambda},\frac{1}{\varepsilon^2}\left(F_{\varepsilon}(v)-\mathrm{Id}\right)\right\rangle &=&\frac{1}{2}\left(\int^0_{-1} 2  t\left\langle \bar\lambda,[f_-,g_{v }]\right\rangle \mathrm{d} t+\,\int^1_{0} 2t\left\langle \bar\lambda,[f_+,g_{v}]\right\rangle \mathrm{d} t+\right.\\
&&\\
&&\left.+\int^1_{-1}\int^{1}_t \left\langle \bar\lambda,[g_{v(\tau)},g_{v(t)}]\right\rangle \mathrm{d} \tau\,\mathrm{d} t\, \right)  +O(\varepsilon).\\
\end{array}
\end{equation}
One can give the following estimate for $O(\varepsilon)$
$$O(\varepsilon)\leq\varepsilon\,\,\mathrm{const}\int^1_{-1}|t||v|^2dt.$$
Hence let us give the following Claim
\begin{claim}
\label{l683}
In order to prove the optimality of the switched curve among all perturbations,  denoting 
\begin{equation}
\label{993}
\begin{array}{l}
J(v)=\int^0_{-1} 2  t\left\langle \bar\lambda,[f_-,g_{v }]\right\rangle \mathrm{d} t+\,\int^1_{0} 2t\left\langle \bar\lambda,[f_+,g_{v}]\right\rangle \mathrm{d} t+\\
\\
+\int^1_{-1}\int^{1}_t \left\langle \bar\lambda,[g_{v(\tau)},g_{v(t)}]\right\rangle \mathrm{d} \tau\,\mathrm{d} t\, +\varepsilon\,\,\mathrm{const}\int^1_{-1}|t||v|^2dt.
\end{array}
\end{equation}
it is enough to prove the following:\\
\textbf{\emph{ Statement}}: \\
There exists $\bar{\varepsilon}>0$ such that $\forall v\neq 0$ and $\forall \varepsilon<\bar\varepsilon$ the following inequality holds
$$J(v)<0.
$$
\end{claim}
\subsection{Result}
The reduction of the problem that we explained in the previous subsection, permits to show the following result.
\begin{theorem}
\label{teosuf3}
Given an affine control system of type (\ref{246h}) with any $n$ and $k=2$, and $\bar\lambda\in \Lambda$ that satisfies (\ref{p265}), if
$$\bar\lambda\perp\mathrm{span}\{f_1(\bar{q}),f_2(\bar{q}),[f_1,f_2](\bar{q})\},$$
then the projection on $M$ of the extremal through $\bar\lambda$ is time-optimal.
\end{theorem}
 \begin{proof}
Let us compute explicitly equation (\ref{993}), in particular it holds
$$\begin{array}{l}
\left\langle \bar{\lambda},[f_-,g_{v}]\right\rangle= -|h_{02}|\sqrt{1-\alpha^2} \left[ \left(\begin{array}{c}
v_1\\
-v_2
\end{array}\right)\cdot a \right]\\ \\
\left\langle \bar{\lambda},[f_+,g_{v}]\right\rangle= |h_{02}|\sqrt{1-\alpha^2} \left[ \left(\begin{array}{c}
v_1\\
v_2
\end{array}\right)\cdot a \right]\\
\\
\left\langle \bar{\lambda},[g_{v(\tau)},g_{v(t)}]\right\rangle=-|h_{12}|\left(v_1(\tau)v_2(t)-v_1(t)v_2(\tau) \right)
\end{array}
$$
For simplicity let us denote
$$\left\lbrace \begin{array}{ll}
V^-:=\left(\begin{array}{c}
v_1\\
-v_2
\end{array}\right)&t<0\\
&\\
V^+:=\left(\begin{array}{c}
v_1\\
v_2
\end{array}\right)&t>0,
\end{array}\right. 
$$
and (\ref{993}) becomes
$$
\begin{array}{l}
J(V^-,V^+)=2|h_{02}|\,\sqrt{1-\alpha^2}\left(\int^0_{-1}   -t\,\, V^-\cdot a\, \mathrm{d} t+\,\int^1_{0} t \,\,V^+\cdot a \,\mathrm{d} t\right)+\\
\\
+\int^1_{-1}\int^{1}_t \left\langle \bar\lambda,[g_{v(\tau)},g_{v(t)}]\right\rangle \mathrm{d} \tau\,\mathrm{d} t+\varepsilon\,\,\mathrm{const}\,\left(\int^1_{0}|t||V^+|^2dt+\int^0_{-1}|t||V^-|^2dt\right).
\end{array}
$$
Moreover, we can consider for $t>0$
$$\left\lbrace \begin{array}{l}
\tilde{V}^-(t):=V^-(-t)\\
\\
\tilde{V}^+(t):=V^+(t),
\end{array}\right. 
$$
then we have
\begin{equation}
\label{997}
\begin{array}{l}
J(\tilde{V}^+,\tilde{V}^-)=2|h_{02}|\,\sqrt{1-\alpha^2}\left(\int^1_{0}   t\,\, \tilde{V}^-\cdot a\, \mathrm{d} t+\,\int^1_{0} t \,\,\tilde{V}^+\cdot a \,\mathrm{d} t\right)+\\
\\
+\int^1_{-1}\int^{1}_t \left\langle \bar\lambda,[g_{v(\tau)},g_{v(t)}]\right\rangle \mathrm{d} \tau\,\mathrm{d} t+\varepsilon\,\,\mathrm{const}\,\left(\int^1_{0}t|\tilde{V}^+|^2dt+\int^1_{0}t|\tilde{V}^-|^2dt\right).
\end{array}
\end{equation}
From equation (\ref{t329}) we have $\left|\tilde{V}^\pm\right|^2\leq -2\tilde{V}^\pm\cdot a$, thus, assuming $\mathrm{const}=|h_{02}|$ it holds
$$\begin{array}{l}
J(\tilde{V}^+,\tilde{V}^-)\leq 2|h_{02}|\,\left(\sqrt{1-\alpha^2}-\varepsilon\right)\left(\int^1_{0}   t\,\, |\tilde{V}^-|^2\, \mathrm{d} t+\,\int^1_{0} t \,\, |\tilde{V}^+|^2 \,\mathrm{d} t\right)+\\
\\
+\int^1_{-1}\int^{1}_t \left\langle \bar\lambda,[g_{v(\tau)},g_{v(t)}]\right\rangle \mathrm{d} \tau\,\mathrm{d} t.
\end{array}
$$
Thus, we prove that it holds the statement of Claim \ref{l683}, if $[f_1,f_2](\bar{q})\in\mathrm{span}\{f_1(\bar{q}),f_2(\bar{q})\}$, indeed we will have $h_{12}=0$, $\alpha=0$ and
\begin{equation}
\label{991}
\begin{array}{l}
J(\tilde{V}^+,\tilde{V}^-)\leq -|h_{02}|\,\left(\sqrt{1-\alpha^2}-\varepsilon\right)\left(\int^1_{0}   t\,\, |\tilde{V}^-|^2\, \mathrm{d} t+\,\int^1_{0} t \,\, |\tilde{V}^+|^2 \,\mathrm{d} t\right)
\end{array}
\end{equation}
is strictly negative if the perturbation $v$ is not null.
 \end{proof}
 \section[Sufficient optimality, with n=3 and k=2]{Sufficient optimality condition with n=3 and k=2}
Finally, let us summarise sufficient optimality results for a system (\ref{246h}) when $n=3$ and $k=2$.\\

We proved the optimality of broken extremals, that passes through $\bar{\lambda}\in \Lambda$ such that $\bar{q}\in \pi(\bar\lambda)$, if
\begin{itemize}
\item $f_0\wedge f_1\wedge f_2\neq 0$ at $\bar{q}$, namely $f_0,f_1,f_2$ are linearly independent at point $\bar{q}$,
\end{itemize}
or
\begin{itemize}
\item  $f_1\wedge f_2\wedge [f_1,f_2]= 0$ at $\bar{q}$, namely those fields are linearly dependent at point $\bar{q}$.
\end{itemize}
It means that each normal extremal that projects in $O_{\bar{q}}$, a neighbourhood of $\bar{q}$ small enough, is optimal. On the other hand, if at point $\bar{q}$ the distribution $\mathrm{span}\{f_1(q),f_2(q)\}$ is not contact, any broken extremal (even abnormal) passing through $\bar\lambda$ is optimal.\\

Among all settings, it remains the case in which 
\begin{itemize}
\item $f_0\wedge f_1\wedge f_2= 0$ and $f_1\wedge f_2\wedge [f_1,f_2]\neq 0$ at point $\bar{q}$,
\end{itemize}
namely, we have a broken abnormal extremal passing through $\bar\lambda$ and the fields $f_1$ and $f_2$ generate a contact distribution at $\bar{q}$.

\end{document}